\documentclass[12pt]{article}
\usepackage{amsmath,amssymb,amsthm,amsfonts,amsbsy, enumerate,graphicx,graphics,epstopdf}
\usepackage[dvips]{epsfig}

\usepackage{verbatim}

\theoremstyle{plain}
\newtheorem{theorem}{Theorem}[section]
\newtheorem{corollary}[theorem]{Corollary}

\newtheorem{lemma}[theorem]{Lemma}

\begin{document}

\title{The spectrum and toughness of regular graphs}
\author{Sebastian M. Cioab\u{a} \footnote{Department of Mathematical Sciences, University of Delaware, Newark, DE 19716-2553, {\tt cioaba@math.udel.edu}} \, and Wiseley Wong \footnote{Department of Mathematics, University of California, San Diego, La Jolla, CA 92093-0112, {\tt w7wong@ucsd.edu}}}
\date{\today}
\maketitle

\begin{abstract}
In 1995, Brouwer proved that the toughness of a connected $k$-regular graph $G$ is at least $k/\lambda-2$, where $\lambda$ is the maximum absolute value of the non-trivial eigenvalues of $G$. Brouwer conjectured that one can improve this lower bound to $k/\lambda-1$ and that many graphs (especially graphs attaining equality in the Hoffman ratio bound for the independence number) have toughness equal to $k/\lambda$. In this paper, we improve Brouwer's spectral bound when the toughness is small and we determine the exact value of the toughness for many strongly regular graphs attaining equality in the Hoffman ratio bound such as Lattice graphs, Triangular graphs, complements of Triangular graphs and complements of point-graphs of generalized quadrangles. For all these graphs with the exception of the Petersen graph, we confirm Brouwer's intuition by showing that the toughness equals $k/(-\lambda_{min})$, where $\lambda_{min}$ is the smallest eigenvalue of the adjacency matrix of the graph.
\end{abstract}

\section{Introduction}

The toughness $t(G)$ of a connected graph $G$ is the minimum of $\frac{|S|}{c(G\setminus S)}$, where the minimum is taken over all vertex subsets $S$ whose removal disconnects $G$, and $c(G\setminus S)$ denotes the number of components of the graph obtained by removing the vertices of $S$ from $G$. A graph $G$ is called $t$-tough if $t(G)\geq t$. Chv\'{a}tal \cite{Chv} introduced this parameter in 1973 to capture combinatorial properties related to the cycle structure of a graph. The toughness of a graph is related to many other important properties of a graph such as Hamiltonicity, and the existence of various factors, cycles or spanning trees and it is a hard parameter to determine exactly (see the survey \cite{Ba}). Two of Chv\'{a}tal conjectures from \cite{Chv} motivated a lot of subsequent work. The first conjecture stated that there exists some $t_0>0$ such that any graph with toughness greater than $t_0$ is Hamiltonian. This conjecture is open at present time and Bauer, Broersma and Veldman \cite{BBV} showed that if such a $t_0$ exists, then it must be at least $9/4$. The second conjecture of Chv\'{a}tal asserted the existence of $t_1>0$ such that any graph with toughness greater than $t_1$ is pancyclic. This was disproved by several authors including Alon \cite{Al}, who showed that there are graphs of arbitrarily large girth and toughness. 
Alon's results relied heavily on the following theorem relating the toughness of a regular graph and its eigenvalues. If $G$ is a connected $k$-regular graph on $n$ vertices, we denote the eigenvalues of its adjacency matrix as follows: $k=\lambda_1>\lambda_2\geq \dots \geq\lambda_n$ and  we let $\lambda=\max(|\lambda_2|,|\lambda_n|)$.
\begin{theorem}[Alon \cite{Al}]\label{AlonThm}
If $G$ is a connected $k$-regular graph, then 
\begin{equation}\label{AlonIneq}
t(G)>\frac{1}{3}\left(\frac{k^2}{k\lambda+\lambda^2}-1\right).
\end{equation}
\end{theorem}
Around the same time and independently, Brouwer \cite{Br} discovered slightly better relations between the toughness of a regular graph and its eigenvalues.
\begin{theorem}[Brouwer \cite{Br}]\label{Brthm}
If $G$ is a connected $k$-regular graph, then
\begin{equation}\label{Brbnd}
t(G)>\frac{k}{\lambda}-2.
\end{equation}
\end{theorem}
Brouwer \cite{Br1} conjectured that the lower bound of the previous theorem can be improved to $t(G) \ge \frac{k}{\lambda}-1$ for any connected $k$-regular graph $G$. This bound would be best possible as there exists regular bipartite graphs with toughness very close to $0$. Brouwer \cite{Br1} mentioned the existence of such graphs, but did not provide any explicit examples. At the suggestion of one of the referees, we briefly describe a construction of such graphs here. Take $k$ disjoint copies of the bipartite complete graph $K_{k,k}$ without one edge, add two new vertices and make each of these new vertices adjacent to one vertex of degree $k-1$ in each copy of $K_{k,k}$ minus one edge. The resulting graph is bipartite $k$-regular and has toughness at most $2/k$ since deleting the two new vertices creates $k$ components.

Liu and Chen \cite{LC} found some relations between the Laplacian eigenvalues and the toughness of a graph. They also improved the eigenvalue conditions of Alon and Brouwer for guaranteeing $1$-toughness.
\begin{theorem}[Liu and Chen \cite{LC}]\label{LCThm}
If $G$ is a connected $k$-regular graph and 
\begin{equation}
\lambda_2<
\begin{cases}
k-1+\frac{3}{k+1}, k \text{ even }\\
k-1+\frac{2}{k+1}, k \text{ odd }
\end{cases}
\end{equation}
then $t(G)\geq 1$.
\end{theorem}

In the first part of our paper, we improve Theorem \ref{AlonThm}, Theorem \ref{Brthm} and Theorem \ref{LCThm}  in certain cases. For small $\tau$, we obtain a better eigenvalue condition than Alon's or Brouwer's that implies a regular graph is $\tau$-tough. We also determine a best possible sufficient eigenvalue condition for a regular graph to be $1$-tough improving the above result of Liu and Chen. We note here that Bauer, van den Heuvel, Morgana and Schmeichel \cite{BHMS,BHMS1} proved that recognizing $1$-tough graphs is an NP-hard problem for regular graphs of valency at least $3$. Our improvements are the following two results. 
\begin{theorem}\label{main0}
Let $G$ be a connected $k$-regular graph on $n$ vertices, $k\geq 3$, with adjacency eigenvalues $k=\lambda_{1}>\lambda_{2}\ge \dots\ge \lambda_{n}$ and edge-connectivity $\kappa'$. If $\tau\leq \kappa'/k$ is a positive number such that $\lambda_{2}(G)<k-\frac{\tau k}{k+1}$, then $t(G)\ge \tau$.
\end{theorem}
\begin{theorem}\label{main1}
If $G$ is a connected $k$-regular graph and 
\begin{equation}\label{main1eq}
\lambda_2(G)<
\begin{cases}
\frac{k-2+\sqrt{k^2+8}}{2} \text{ when $k$ is odd} \\
\frac{k-2+\sqrt{k^2+12}}{2} \text{ when $k$ is even}
\end{cases}
\end{equation}
then $t(G)\geq 1$.
\end{theorem}
The proofs of Theorem \ref{main0} and Theorem \ref{main1} are similar to the one of Liu and Chen \cite{LC} and are contained in Section \ref{main0main1}. We show that Theorem \ref{main1} is best possible in the sense that for each $k\geq 3$, we construct examples of $k$-regular graphs whose second largest eigenvalue equals the right hand-side of inequality \eqref{main1eq}, but whose toughness is less than $1$. These examples are described in Section \ref{examples}. Our examples are regular graphs of diameter $4$ and their existence also answers a question of Liu and Chen \cite[p. 1088]{LC} about the minimum possible diameter of a regular graph with toughness less than $1$.

In \cite{Br1}, Brouwer also stated that he believed that $t(G)=\frac{k}{\lambda}$ for many graphs $G$. Brouwer's reasoning hinged on the fact that a connected $k$-regular graph $G$ with $n$ vertices attaining equality in the Hoffman ratio bound (meaning that the independence number $\alpha(G)$ of $G$ equals  $\frac{n(-\lambda_{min})}{k-\lambda_{min}}$; see e.g. \cite[Chapter 3]{BH} or \cite[Chapter 9]{GR}) and having $\lambda=-\lambda_{min}$, is likely to have toughness equal to $k/\lambda=k/(-\lambda_{min})$. Brouwer argued that for such a graph $G$, $k/\lambda$ is definitely an upper bound for the toughness, (as one can take $S$ to be the complement of an independent set of maximum size $\frac{n\lambda}{k-\lambda}$ and then $t(G)\leq \frac{|S|}{c(G\setminus S)}=k/\lambda$) and Brouwer suggested that for many such graphs $k/\lambda$ is the exact value of $t(G)$.

In the second part of the paper, we determine the exact value of the toughness of several families of strongly regular graphs attaining equality in Hoffman ratio bound, namely the Lattice graphs, the Triangular graphs, the complements of the Triangular graphs and the complements of the point-graphs of generalized quadrangles. Moreover, for each graph $G$ above, we determine the disconnecting sets of vertices $S$ such that $\frac{|S|}{c(G\setminus S)}$ equals the toughness of $G$. We show that for all these graphs except the Petersen graph, the toughness equals $k/(-\lambda_{min})$, where $k$ is the degree of regularity and $\lambda_{min}$ is the smallest eigenvalue of the adjacency matrix. These results are contained in Section \ref{tsrg}.
In Subsection \ref{Lattice}, we prove that the Lattice graph $L_2(v)$ has toughness $v-1=k/(-\lambda_n)$ for $v\geq 2$. The Lattice graph $L_2(v)$ is the line graph of $K_{v,v}$, and is also known as the Hamming graph $H(2,v)$; it is a strongly regular graph with parameters $(v^2,2v-2,v-2,2)$. In Subsection \ref{triangular}, we prove that the Triangular graph $T_v$ has toughness $v-2=k/(-\lambda_n)$ for $v\geq 4$. The graph $T_v$ is the line graph of $K_v$; it is a strongly regular graph with parameters $\left(\binom{v}{2},2v-4,v-2,4\right)$. We remark here that these two results can be also deduced by combining a theorem of Matthews and Sumner \cite{MS} stating that the toughness of a connected noncomplete graph not containing any induced $K_{1,3}$, equals half of its vertex-connectivity, with a theorem of Brouwer and Mesner \cite{BM} stating that the vertex-connectivity of a connected strongly regular graph equals its degree. Our proofs are self-contained, elementary and we also determine all the disconnecting sets $S$ such that $\frac{|S|}{c(G\setminus S)}=t(G)$ for every $G$ in the family of Lattice graphs, Triangular graphs, complements of Triangular graphs or complements of point-graphs of generalized quadrangles. In Subsection \ref{compltriangular}, we prove that the complement $\overline{T_v}$ of the Triangular graph $T_v$ has toughness $v/2-1=k/(-\lambda_{min})$ for $v>5$. When $v=5$, the graph $\overline{T_5}$ is the Petersen graph whose toughness is $4/3<3/2=k/(-\lambda_{min})$. The complement $\overline{T_v}$ of the Triangular graph $T_v$ is a strongly regular graph with parameters $\left({v\choose 2},{v-2\choose 2},{v-4\choose 2},{v-3\choose 2}\right)$. In Section \ref{gqst}, we prove that the complement of the point-graph of a generalized quadrangle of order $(s,t)$ has toughness $st=k/(-\lambda_{min})$.  A generalized quadrangle is a point-line incidence structure such that any two points are on at most one line and if $p$ is  a point not on a line $L$, then there is a unique point on $L$ that is collinear with $p$. If every line contains $s+1$ points and every point is contained in $t+1$ lines, then we say the generalized quadrangle has order $(s,t)$. The complement of the point-graph of a generalized quadrangle is the graph with the points of the quadrangle as its vertices, with two points adjacent if and only if they are not collinear. This graph is a strongly regular graph with parameters $((s+1)(st+1), s^2t, t(s^2+1)-s(t+1),st(s-1))$. 


\section{Proofs of Theorem \ref{main0} and Theorem \ref{main1}}\label{main0main1}

We first give a short proof of Theorem \ref{main0}.
\begin{proof}[Proof of Theorem \ref{main0}]
We prove the contrapositive. Assume $\tau\leq \kappa'/k\leq 1$ is a positive number and that $G$ is a connected $k$-regular graph such that $t(G)<\tau$. Then there exists $S\subset V(G)$ of size $s$ such that $r=c(G\setminus S)>s/\tau\geq s$ (as $\tau\in (0,1]$). We denote by $H_1,\dots,H_r$ the components of $G\setminus S$. For $1\leq i\leq k$, let $n_i$ be the number of vertices in $H_i$ and let $t_i$ be the number of edges between $H_i$ and $S$. The number of edges between $S$ and $H_1\cup \dots \cup H_r$ is $t_1+\dots +t_r\leq ks<k\tau r$. Also, $t_i\geq \kappa'$ for $1\leq i\leq k$.

By pigeonhole principle, there exist at least two $t_i$'s such that $t_i\leq \tau k$. This is true since otherwise, if at least $r-1$ of the $t_i$'s are greater than $\tau k$, then $k\tau r>t_1+\dots +t_r\geq \kappa'+(r-1)\tau k=k\tau r+\kappa'-\tau k\geq k\tau r$, a contradiction. Without loss of generality, assume that $\max(t_1,t_2)\leq \tau k<k$. If $n_1\leq k$, then $t_1\geq n_1(k-n_1+1)\geq k$ which is a contradiction. Thus, $n_1\geq k+1$ and by a similar argument $n_2\geq k+1$. For $i\in \{1,2\}$, the average degree of the component $H_i$ is $\frac{kn_i-t_i}{n_i}=k-\frac{t_i}{n_i}\geq k-\frac{\tau k}{k+1}$. As the largest eigenvalue of the adjacency matrix of a graph is at least its average degree (see e.g. \cite[Proposition 3.1.2]{BH}), we obtain that $\min(\lambda_1(H_1),\lambda_1(H_2))\geq k-\frac{\tau k}{k+1}$. By eigenvalue interlacing (see e.g. \cite[Section 2.5]{BH}), we obtain
\begin{equation}
\lambda_2(G)\geq \lambda_2(H_1\cup H_2)\geq \min(\lambda_1(H_1),\lambda_1(H_2))\ge k-\frac{\tau k}{k+1}.
\end{equation}
This finishes our proof.
\end{proof}

In order to prove Theorem \ref{main1}, we need to do some preliminary work. If $t$ is a positive even integer, denote by $M_t$ the union of $t/2$ disjoint $K_2$s. If $G$ and $H$ are two vertex disjoint graphs, the join $G\vee H$ of $G$ and $H$ is the graph obtained by taking the union of $G$ and $H$ and adding all the edges between the vertex set of $G$ and the vertex set of $H$. The complement of a graph $G$ is denoted by $\overline{G}$.

Let $k\geq 3$ be an integer. Denote by $\mathcal{X}(k)$ the family of all connected irregular graphs with maximum degree $k$, order $n\geq k+1$, size $e$ with $2e\geq kn-k+1$ that have at least $2$ vertices of degree $k$ when $k$ is odd and at least $3$ vertices of degree $k$ when $k$ is even. Denote by $\theta(k)$ the minimum spectral radius of a graph in $\mathcal{X}(k)$.
\begin{lemma}
For $k\geq 3$, define
\begin{equation}
X_k=
\begin{cases}
\overline{M_{k-1}}\vee K_2 \text{ if $k$ is odd}\\
\overline{M_{k-2}}\vee K_3 \text{ if $k$ is even}
\end{cases}
\end{equation}
Then 
\begin{equation}
\lambda_1(X_k)=
\begin{cases}
\frac{k-2+\sqrt{k^2+8}}{2} \text{ if $k$ is odd}\\
\frac{k-2+\sqrt{k^2+12}}{2} \text{ if $k$ is even}.
\end{cases}
\end{equation}
\end{lemma}
\begin{proof}
If $k\geq 3$ is odd, then the partition of the vertex set of $X_k$ into the $2$ vertices of degree $k$ and the $k-1$ vertices of degree $k-1$ is an equitable partition whose quotient matrix is $\begin{bmatrix} 1 & k-1\\ 2 & k-3\end{bmatrix}$. The largest eigenvalue of $X_k$ equals the largest eigenvalue of this quotient matrix which is $\frac{k-2+\sqrt{k^2+8}}{2}$ (see e.g. \cite[Section 2.3]{BH}). If $k$ is even, the proof is similar.
\end{proof}

It is clear that $X_k\in \mathcal{X}_k$ for any $k\geq 3$ and consequently, $\theta(k)\leq \lambda_1(X_k)$. In the next lemma, we show that actually $\theta(k)=\lambda_1(X_k)$ and that $X_k$ is the only graph in $\mathcal{X}_k$ whose largest eigenvalue is $\theta(k)$.

\begin{lemma}\label{minXk}
Let $k\geq 3$ be an integer. If $X\in \mathcal{X}_k\setminus \{X_k\}$ then $\lambda_1(X)>\lambda_1(X_k)$ and thus,
\begin{equation}
\theta(k)=\begin{cases}
\frac{k-2+\sqrt{k^2+12}}{2} \text{ if $k$ is even}\\
\frac{k-2+\sqrt{k^2+8}}{2} \text{ if $k$ is odd}
\end{cases}
\end{equation}
\end{lemma}
\begin{proof}
When $k$ is even, this was proved in \cite[Theorem 2]{CGH} so for the rest of the proof, we assume that $k$ is odd.
Let $X\in \mathcal{X}_k\setminus \{X_k\}$. We will prove that $\lambda_1(X)>\lambda_1(X_k)$.

Assume first that $n\geq k+2$, where $n$ is the number of vertices of $X$. As $X\in \mathcal{X}_k$, we know that $2e\geq kn-k+1$, where $e$ is the number of edges of $X$. Because the largest eigenvalue of $X$ is at least the average degree of $X$, after some straightforward calculations we obtain that
\begin{equation}
\lambda_{1}(X) \ge \frac{2e}{n}=k-\frac{k-1}{n}\ge k-\frac{k-1}{k+2}>\frac{k-2+\sqrt{k^{2}+8}}{2}=\lambda_1(X_k).
\end{equation}
Assume now that $n=k+1$. If $2e>kn-k+1$, then, since $k$ is odd, we must have $2e\geq kn-k+3$. We obtain that 
\begin{equation}
\lambda_1(X)\geq \frac{2e}{n}\geq k-\frac{k-3}{k+1}>\frac{k-2+\sqrt{k^2+8}}{2}=\lambda_1(X_k).
\end{equation}

The only possible case remaining is $n=k+1$ and $2e=kn-k+1$. Let $V_{1}$ be a subset of two vertices of degree $k$ and let $V_{2}$ the complement of $V_1$. The quotient matrix of the partition of the vertex set into $V_1$ and $V_2$ is
\begin{equation}
B=\begin{bmatrix}
1 & k-1\\
2&k-3
\end{bmatrix}.
\end{equation}
Eigenvalue interlacing \cite[Section 2.5]{BH} gives 
\begin{equation}\label{int}
\lambda_1{(X)}\ge\lambda_{1}(B)=\frac{k-2+\sqrt{k^2+8}}{2}= \lambda_1{(X_{k})},
\end{equation}
with equality if and only if the partition $V_1\cup V_2$ is equitable. This happens if and only if $X=X_k$. As $X\in \mathcal{X}_k\setminus \{X_k\}$, we get that $\lambda_1(X)>\frac{k-2+\sqrt{k^2+8}}{2}=\lambda_1(X_k)$ which finishes our proof.
\end{proof}

We are ready now to prove Theorem \ref{main1}. The proof is similar to the one of Theorem \ref{main0}.
\begin{proof}[Proof of Theorem \ref{main1}]
We prove the contrapositive. We show that $t(G)<1$ implies that $\lambda_2(G)\geq \theta(k)$. If $t(G)<1$, then there exists an $S\subset V(G)$ of size $s$ such that $r=c(G\setminus S)>|S|=s$. As in the proof of Theorem \ref{main0}, let $H_1, H_2, ..., H_r$ be the components of $G\setminus S$. For $1\leq i\leq r$, denote by $t_i\ge 1$ denote the number of edges between $S$ and $H_i$ and by $n_i$ the number of vertices of $H_i$. Then $\sum_{i=1}^r t_i=e(S,V\setminus S) \le ks$. We claim there exists two $t_i$'s such that $t_i <k$.  Indeed, if there exists at most one such $t_i$, then at least $r-1$ of the $t_i$'s are greater than $k$ and thus, $ks\geq \sum_{i=1}^r t_i\ge (r-1)k+1>ks$ which is a contradiction. Thus, there exist at least two $t_i$'s (say $t_1$ and $t_2$) such that $\max(t_1,t_2)<k$. If $n_1\leq k$, then $t_1\geq n_1(k-n_1+1)\geq k$ which is a contradiction. Thus, $n_1\geq k+1$ and by a similar argument $n_2\geq k+1$. As $t_i\leq k-1$ for $i\in \{1,2\}$, we get $2e(H_i)=kn_i-t_i\geq kn_i-k+1$. Also, if $k$ is even, then $2e(H_i)=kn_i-t_i$ implies that $t_i$ is even. For $i\in \{1,2\}$, this means $t_i\leq k-2$ as $t_i<k$.  Hence, $\max(t_1,t_2)\leq k-2$ when $k$ is even and $\max(t_1,t_2)\leq k-1$ when $k$ is odd. This implies $H_i$ contains at least two vertices of degree $k$ when $k$ is odd and at least $3$ vertices of degree $k$ when $k$ is even. Thus, $H_i\in \mathcal{X}_k$ for $i\in \{1,2\}$. By Lemma \ref{minXk}, we get that $\min(\lambda_1(H_1),\lambda_1(H_2))\geq \theta(k)$. By Cauchy eigenvalue interlacing \cite[Section 2.5]{BH}, we obtain $\lambda_2(G)\geq \lambda_2(H_1\cup H_2)\geq \min(\lambda_1(H_1),\lambda_1(H_2))\geq \theta(k)$ which finishes our proof.
\end{proof}

The following is an immediate corollary of Theorem \ref{main1} that can be used to determine the toughness of the $n$-dimensional cube $H(n,2)$ and of even cycles. 
\begin{corollary}{\label{tcor}}
If $G$ is a bipartite $k$-regular graph with
\begin{equation}
\lambda_{2}(G)<
\begin{cases}
\frac{1}{2}\left(k-2+\sqrt{k^{2}+8}\right)& k \text{ is odd}\\
\frac{1}{2}\left(k-2+\sqrt{k^{2}+12}\right)& k \text{ is even},
\end{cases}
\end{equation}
then $t(G) = 1$.  
\end{corollary}
\begin{proof}
Let $S$ be the set of vertices of one part of the bipartition.  Then $\frac{|S|}{c(G\setminus S)}=1$, and so $t(G)\le 1$.  Theorem \ref{main1} implies that $t(G)=1$. 
\end{proof}
We believe that Corollary \ref{tcor} can be improved by proving a similar result to Lemma \ref{minXk} for bipartite graphs.

\section{Examples showing Theorem \ref{main1} is best possible}\label{examples}

In this section, we construct $k$-regular graphs with second largest eigenvalue $\theta(k)$ that are not $1$-tough, for every $k\geq 3$. This shows that the bound contained in Theorem \ref{main1} is best possible.


\begin{figure}[h!]
\centering
\includegraphics[scale=0.7]{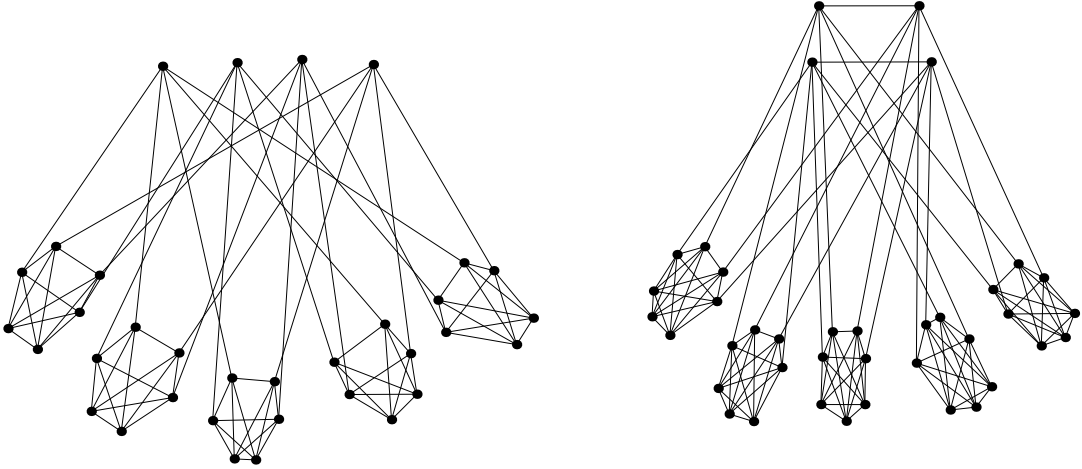}
\caption{The extremal examples of Theorem \ref{main1} for $k=5$ and $k=6$.}
\label{extremal56}
\end{figure}


If $k\geq 3$ is odd, consider $k$ pairwise vertex disjoint copies of $\overline{M}_{k-1}\vee K_{2}$. Consider also a set $T$ of $k-1$ isolated vertices. The graph $\mathcal{G}_k$ is obtained by adding a matching of size $k-1$ between the vertex set of $T$ and the $k-1$ vertices of degree $k-1$ in each of the $k-1$ copies of $\overline{M}_{k-1}\vee K_{2}$. The graph $\mathcal{G}_k$ is a connected $k$-regular graph on $k(k+1)+k-1=k^2+2k-1$ vertices. Figure \ref{extremal56} contains a picture of $\mathcal{G}_5$. By choosing $S\subset V(\mathcal{H}_{k})$ as the independent set $T$ of $k-1$ vertices, the graph $\mathcal{G}_k$ has toughness at most $\frac{k-1}{k}<1$.

We claim $\lambda_{2}(\mathcal{G}_{k}) = \frac{1}{2}\left(k-2+\sqrt{k^{2}+8}\right)$.  The vertices of $G= G_k= \overline{M}_{k-1}\vee K_{2}$ can be ordered so that its adjacency matrix is partitioned in the form
\vspace{.2in} 

$$A(G) = \begin{bmatrix}
A(\overline{M}_{k-1})&J\\
J^T&A(K_2)
\end{bmatrix},$$
where $J$ denotes an all-ones matrix of appropriate size.  The $n=k^2+2k-1$ vertices of $\mathcal{G}_k$ can be ordered so that its adjacency matrix is partitioned in the form 
$$A\left(\mathcal{G}_k \right) = \begin{bmatrix}
0&B&B&\cdots&B\\
B^T&A(G)&0&\cdots&0\\
B^T&0&A(G)&\cdots&0\\
\vdots&\vdots&\vdots&\ddots&\vdots\\
B^T&0&0&\cdots&A(G)
\end{bmatrix},$$
where 0 denotes a zero matrix of appropriate size, and $B=\begin{bmatrix}I&0\end{bmatrix}$, where $I$ is an identity matrix of order $k-1$.

The eigenvectors of $G_k$ that are constant on each part of its  2-part equitable partition have eigenvalues given by the quotient matrix
$$\begin{bmatrix}
k-3&2\\
k-1&1
\end{bmatrix}.$$
These are $\frac{1}{2}\left(k-2\pm\sqrt{k^2+8}\right)$, and the positive eigenvalue is $\lambda_1(G) = \theta(k)$ from Lemma \ref{minXk}.  Let $x$ and $y$ be eigenvectors of $G$ associated with these two eigenvalues.  

If $u$ is a column eigenvector of $A(G)$, consider the $k-2$ column $n$-vectors

$$\begin{bmatrix}
0\\
u\\
-u\\
\vdots\\
\vdots\\
0
\end{bmatrix},
\begin{bmatrix}
0\\
u\\
0\\
-u\\
\vdots\\
0
\end{bmatrix}, \cdots\cdots,
\begin{bmatrix}
0\\
u\\
0\\
0\\
\vdots\\
-u
\end{bmatrix},$$
where the zero vectors are of appropriate size, and the first zero vector is always present (corresponding to the $k-1$ isolated vertices).  It is straightforward to check that these $k-1$ vectors are linearly independent eigenvectors of $A(\mathcal{G}_k)$ with the same eigenvalues as $u$.  Thus each eigenvalue of $A(G)$ of multiplicity $m$ yields an eigenvalue of $A(\mathcal{G}_k)$ of multiplicity at least $m(k-1)$.  In particular, taking $u=x$ and $u=y$, we see that the two eigenvalues of $A(G)$,  $\frac{1}{2}\left(k-2\pm\sqrt{k^2+8}\right)$, yield eigenvalues of $A(\mathcal{G}_k)$ of multiplicity at least $k-1$ each. Thus, $\lambda_1(G) = \theta(k)$ is an eigenvalue of $\mathcal{G}_k$ with multiplicity at least $k-1$.

Now consider the $(2k+1)$-part equitable partition of $\mathcal{G}_k$ obtained by extending the 2-part partition of the $k$ copies of ${G}_k$ in $\mathcal{G}_k$.  Let $W$ be the space consisting of $n$-vectors that are constant on each part of the partition, which has dimension $2k+1$.  Note that each of the $2(k-1)$ independent eigenvectors of $\mathcal{G}_k$ inherited from the eigenvectors $x$ and $y$ of $G_k$ are in $W$.  The natural 3-part equitable partition of $\mathcal{G}_k$ has quotient matrix
$$\begin{bmatrix}
0&k&0\\
1&k-3&2\\
0&k-1&1
\end{bmatrix},$$
with eigenvalues $k$ and $-1\pm\sqrt{2}$.  Their corresponding eigenvectors lift to eigenvectors of $\mathcal{G}_k$ in $W$ with the same eigenvalues, and these three eigenvalues are different from those above.  Thus the three lifted eigenvectors, together with the previous $2(k-1)$ eigenvectors of $\mathcal{G}_k$ inherited from $G$, form a basis for $W$.

The remaining eigenvectors in a basis of eigenvectors for $\mathcal{G}_k$ may be chosen orthogonal to the vectors in $W$.  They may be chosen to be orthogonal to the characteristic vectors of the parts of the $(2k+1)$-part partition because the characteristic vectors are also a basis for $W$.  Therefore, they will be some of the eigenvectors of the matrix $A(\hat{\mathcal{G}_k})$ obtained from $A(\mathcal{G})$ by replacing each all-ones block in each diagonal block of $A(G)$ by an all-zeros matrix.  But $A(\hat{\mathcal{G}_k})$ is the adjacency matrix of a graph $\hat{\mathcal{G}_k}$ with $k+1$ connected components, one of which is the graph $\mathcal{G}'$ obtained by attaching $k$ copies of $\overline{M_{k-1}}$ to a set of $k-1$ independent vertices by perfect matchings.  Each of the remaining $k$ components is a copy of $K_2$.  It follows that the greatest eigenvalue of $\hat{\mathcal{G}_k}$ is that of a component of $\mathcal{G}'$.  Because $\mathcal{G}'$ has a 2-part equitable partition with quotient matrix
$$\begin{bmatrix}
0&k\\
1&k-3
\end{bmatrix},$$
its greatest eigenvalue is $\frac{1}{2}\left(k-3+\sqrt{k^2-6k+25}\right)$, which is less than $\frac{1}{2}\left(k-2+\sqrt{k^2+8}\right)= \theta(k)$.


If $k\geq 4$ is even, consider $k-1$ vertex disjoint copies of $\overline{M}_{k-2}\vee K_{3}$. Consider also a vertex disjoint copy $T$ of $M_{k-2}$. The graph $\mathcal{H}_{k}$ is obtained by adding a matching of size $k-2$ between the vertices in $T$ and the $k-2$ vertices of degree $k-1$ in each of the $k-1$ copies of $\overline{M}_{k-2}\vee K_{3}$. Figure \ref{extremal56} contains a picture of the graph $\mathcal{H}_6$. The graph $\mathcal{H}_k$ is a connected $k$-regular graph on $(k-1)(k+1)+k-2=k^2+k-3$ vertices. By choosing $S\subset V(\mathcal{H}_{k})$ as the vertices in $T$, it is easy to check that $\mathcal{H}_{k}$ has toughness at most $\frac{k-2}{k-1}<1$.

Similarly to the case $k$ odd, we can show that $\lambda_2(\mathcal{H}_{k})=  \frac{1}{2}\left(k-2+\sqrt{k^{2}+12 }\right).$  The full details of the proof of this fact can be found in the second author's Ph.D. Thesis \cite{WongThesis}. At the suggestion of one of the referees, we omit these details here.

\section{The toughness of some strongly regular graphs attaining equality in the Hoffman ratio bound}\label{tsrg}

In this section, we determine the toughness of several families of strongly regular graphs attaining equality in the Hoffman ratio bound. Except for the Petersen graph and the Triangular graph $T_v$ with $v$ odd, every graph described in the next four subsections attains equality in the Hoffman ratio bound. For each such graph, we have an upper bound of $k/(-\lambda_{min})$ for the toughness. In the next subsections, we prove that $k/(-\lambda_{min})$ is also a lower bound for the toughness of such graphs.


\subsection{The toughness of the Lattice graphs}{\label{Lattice}}

In this section, we determine the exact value of the toughness of the Lattice graph $L_2(v)$ for $v\geq 2$. The Lattice graph $L_2(v)$ is the line graph of the complete bipartite graph $K_{v,v}$ and is a strongly regular graph with parameters $(v^2,2v-2,v-2,2)$.  The spectrum of the adjacency matrix of $L_2(v)$ is $(2v-2)^{(1)}, (v-2)^{(2v-2)}, (-2)^{(v^2-2v+1)}$, where the exponents denote the multiplicities of the eigenvalues. 
\begin{theorem}\label{L2v}
For $v\geq 2, t(L_2(v)) = v-1$. Moreover, if $S$ is a subset of vertices of $L_2(v)$ such that $\frac{|S|}{c(G\setminus S)}=v-1$, then $S$ is the complement of a maximum independent set of size $v$, or $S$ is the neighborhood of some vertex of $L_2(v)$.
\end{theorem}  
\begin{proof}
Let $S\subset V(L_2(v))$ be a disconnecting set of vertices and let $r=c(G\setminus S)$ and $s=|S|$. We show that $\frac{|S|}{c(L_2(v)\setminus S)}\geq v-1$ with equality if and only if $S$ the complement of a maximum independent set of size $v$ or $S$ is the neighborhood of some vertex of $L_2(v)$. 

Consider the vertices of $L_2(v)$ arranged in a $v\times v$ grid, with two vertices adjacent if and only if they are in the same row or the same column. We may assume that the components of $L_2(v)\setminus S$ are pairwise row disjoint and column disjoint $h_i\times w_i$ rectangular blocks for $1\leq i\leq r$. Indeed, if some component $\mathcal{C}$ with maximum height $h$ and width $w$ was not a rectangular block, we can redefine this component as the rectangular block of height $h$ and width $w$ that contains $\mathcal{C}$. This will not affect any other components in $G\setminus S$ and will decrease $|S|$, and thus will decrease $\frac{|S|}{c(G\setminus S)}$. Obviously, any two vertices from distinct components must be in different rows and different columns. Therefore $2\le r\le v$.   

We may also assume that $\sum_{i=1}^r h_i=\sum_{i=1}^r w_i=v$. Otherwise, if $S$ contains an entire row $r$ (or column $c$) of vertices, we can simply add $w_i$ (or $h_i$) vertices from that row (column) to one of the components $h_i \times w_i$ of $L_2(v)\setminus S$. This will decrease $|S|$ and has no effect on $c(G\setminus S)$, and therefore it will decrease $\frac{|S|}{c(G\setminus S)}$. These facts imply that
\begin{align*}
|S|&=v^2-\sum_{i=1}^r h_iw_i \ge v^2 -\frac{\sum_{i=1}^r h_i^2 + \sum_{i=1}^r w_i^2}{2}\ge v^2- (r-1)-(v-r+1)^2\\
&=(r-1)(2v-r)=r(v-1)+(r-2)(v-r)\\
&\geq r(v-1),
\end{align*}
where the last inequality is true since $2\le r\le v$. This proves $\frac{|S|}{c(G\setminus S)}\geq v-1$. Equality happens above if and only if (after an eventual renumbering of the components) $h_i=w_i=1$ for each $1\leq i\leq r-1, h_r=w_r=v-r+1$ and $r=2$ or $r=v$. This means $S$ is the neighborhood of some vertex of $L_2(v)$ (when $r=2$) or $S$ is the complement of an independent set of size $v$ (when $r=v$). This finishes our proof.
\end{proof}

\subsection{The toughness of the Triangular graphs}{\label{triangular}}

In this section, we determine the exact value of the toughness of the Triangular graph $T_v$ for $v\geq 4$. The graph $T_v$ is the line graph of $K_v$; it is a strongly regular graph with parameters $\left(\binom{v}{2}, 2v-4, v-2, 4\right)$. The spectrum of $T_v$ is $(2v-4)^{(1)}, (v-4)^{(v-1)}, (-2)^{\left((v^2-3v)/2\right)}$. 
\begin{theorem}\label{Tv}
For $v\geq 4, t(T_v)=v-2$. Moreover, $S$ is a disconnecting set of vertices of $T_v$ such that $\frac{|S|}{c(T_v\setminus S)}=v-2$ if and only if $S$ is the neighborhood of a vertex of $T_v$ or $S$ is the complement of a maximum independent set of size $\frac{v}{2}$ (when $v$ is even).
\end{theorem}  
\begin{proof}
The Triangular graph $T_v$ has independence number $\alpha(T_v)=\lfloor v/2\rfloor$ and attains equality in the Hoffman ratio bound when $v$ is even as $\frac{n(-\lambda_{min})}{k-\lambda_{min}}=\frac{{v\choose 2}2}{2v-4+2}=v/2$. By our discussion at the beginning of the paper, this implies $t(T_v)\leq k/(-\lambda_{min})=v-2$ when $v$ is even. If $S$ is the neighborhood of some vertex of $T_v$, then $T_v\setminus S$ consists of two components: one isolated vertex and a graph isomorphic to $T_{v-2}$. This shows that $t(T_v)\leq (2v-4)/2=v-2$ regardless of the parity of $v$.

Let $S\subset V(T_v)$ be a disconnecting set of vertices and let $r=c(T_v\setminus S)$ and $s = |S|$. We prove that $s/r\geq v-2$ with equality if and only if $S$ the complement of a maximum independent set of size $v$ (when $v$ is even) or $S$ is the neighborhood of some vertex of $T_v$. The vertices of the graph $T_v$ are the $2$-subsets of $[v]=\{1,\dots,v\}$, and two such subsets are adjacent if and only if they share exactly one element of $[v]$. Thus, $2\le r \le \lfloor\frac{v}{2}\rfloor$. For any component $\mathcal{C}$ in $T_v\setminus S$ whose vertices are $2$-subsets of some subset $A\subset [v]$ (where $|A|$ is minimum with this property), we can assume that $\mathcal{C}$ contains all the ${|A| \choose 2}$ $2$-subsets of $A$. This is because including in $\mathcal{C}$ all the $2$-subsets of $|A|$ decreases $|S|$ and does not change $c(T_v\setminus S)$, therefore decreasing $\frac{|S|}{c(T_v\setminus S)}$.

Thus, without loss of generality we may assume that there exist integers $0=a_0<a_1<a_2<\dots <a_{r-1}<a_r=v$ such that $a_j-a_{j-1}\geq 2$, and the vertex set of the $j$-th component of $T_v\setminus S$ is formed by all the $2$-subsets of $\{a_{j-1}+1,a_{j-1}+2,\dots,a_j\}$ for $1\leq j\leq r$. By letting $n_j=a_j-a_{j-1}\geq 2$, we see $\sum_{j=1}^r n_j =v$. Thus,
\begin{align*}
|S|&=\binom{v}{2}-\sum_{j=1}^r \binom{n_j}{2} \ge \binom{v}{2}-\left[\binom{v-2(r-1)}{2}+r-1\right]\\
&=(r-1)(2v-2r)=r(v-2)+(r-2)(v-2r)\\
&\geq r(v-2)
\end{align*}
where the last inequality is true since $2\le r\le \lfloor \frac{v}{2}\rfloor$. The first equality occurs when the components of $T_v\setminus S$ are $r-1$ isolated vertices and a component consists of all the $2$-subsets of a subset of size $v-2(r-1)$ of $[v]$. The final equality occurs when $r=2$ or $r=\frac{v}{2}$. If $r=2$, equality occurs when $S$ is the neighborhood of a vertex. If $r=\frac{v}{2}$, equality occurs when $v$ is even and all the components of $T_v\setminus S$ are isolated vertices, meaning $S$ is the complement of a maximum independent set.
\end{proof}

\subsection{The toughness of the complements of the Triangular graphs}\label{compltriangular}

The complement $\overline{T_v}$ of the Triangular graph $T_v$ has the $2$-subsets of a set $[v]$ with $v$ elements as its vertices, with two subsets being adjacent if and only they are disjoint. It is a strongly regular graph with parameters $\left({v\choose 2}, {v-2\choose 2}, {v-4\choose 2}, {v-3\choose 2}\right)$. The spectrum of $\overline{T_v}$ is ${v-2\choose 2}^{(1)}, 1^{\left((v^2-3v)/2\right)},(3-v)^{(v-1)}$. When $v=5$, $\overline{T_5}$ is the Petersen graph. If $S$ is a maximum independent set, say $S=\{ \{1,2\},\{1,3\}, \{1,4\}, \{1,5\}\}$, then $\frac{|S|}{c(\overline{T_5}\setminus S)} = \frac{4}{3}$ which shows $t(\overline{T_5})\leq \frac{4}{3}$. By a simple case analysis, it can be shown that actually $t(\overline{T_5})=4/3<3/2$.
\begin{theorem}\label{cTv}
For $v\geq 6$, $t(\overline{T_v})=v/2-1=k/(-\lambda_{min})$.  Furthermore, if $S$ is a disconnecting set of vertices in $\overline{T_v}$ such that $\frac{|S|}{c(\overline{T_v}\setminus S)}=v/2-1$, then $S$ is the complement of an independent set of maximum size in $\overline{T_v}$.
\end{theorem}
\begin{proof}
Let $S\subset V(\overline{T_v})$ be a disconnecting set of vertices and let $r=c(\overline{T_v}\setminus S)$ and $s=|S|$. We prove that $s/r\geq v/2-1$ with equality if and only if $S$ the complement of a maximum independent set of size $v-1$. As $S$ is a disconnecting set of vertices, we get that $|S|\geq {v-2\choose 2}$ with equality if and only if $S$ is the neighborhood of some vertex of $\overline{T_v}$ (see \cite{BM}). If $|S|={v-2\choose 2}$, then $\overline{T_v}\setminus S$ has two components: one isolated vertex and a subgraph of $\overline{T_v}$ isomorphic to the complement of a perfect matching in a complete bipartite graph on $2v-4$ vertices. In this case, we get that $\frac{|S|}{c(\overline{T_v}\setminus S)}=\frac{\binom{v-2}{2}}{2} =\frac{(v-2)(v-3)}{4}>\frac{v}{2}-1$ as $v\geq 6$.

If $|S|>\binom{v-2}{2}$, then $\frac{|S|}{c(\overline{T_v}\setminus S)}\geq \frac{\binom{v-2}{2}+1}{r}=\frac{v^2-5v+8}{2r}$. If $r\leq v-3$, then $\frac{v^2-5v+8}{2r}>v/2-1$, and we are done in this case. Also, if $r=v-1$, then the set formed by picking a single vertex from each component of $\overline{T_v}\setminus S$ is an independent set of maximum size $v-1$, and therefore it consists of all $2$-subsets containing some fixed element of $[v]$. This implies that $\overline{T_v}\setminus S$ is actually an independent set of size $v-1$ and we obtain equality $\frac{|S|}{c(\overline{T_v}\setminus S)}=\frac{{v\choose 2}-(v-1)}{v-1}=v/2-1$ in this case. If $r=v-2>3$, then the set formed by picking a single vertex from each component of $\overline{T_v}\setminus S$ is an independent set of size $v-2>3$. This independent set consists of $v-2$ $2$-subsets containing some fixed element of $[v]$. This implies that the components of $\overline{T_v}\setminus S$ must be $v-2$ isolated vertices and therefore, $\frac{|S|}{c(\overline{T_v}\setminus S)}=\frac{{v\choose 2}-(v-2)}{v-2}>v/2-1$. This finishes our proof.
\end{proof}


\subsection{The toughness of the complements of the point-graphs of generalized quadrangles}\label{gqst}

A generalized quadrangle is a point-line incidence structure such that any two points are on at most one line and if $p$ is a point not on a line $L$, then there is a unique point on $L$ that is collinear with $p$. If every line contains $s+1$ points and every point is contained in $t+1$ lines, then we say the generalized quadrangle has order $(s,t)$. The point-graph of a generalized quadrangle is the graph with the points of the quadrangle as its vertices, with two points adjacent if and only if they are collinear. The point-graph of a generalized quadrangle of order $(s,t)$ is a strongly regular graph with parameters $((s+1)(st+1),s(t+1),s-1,t+1)$ (see e.g. \cite[Section 10.8]{GR}). The complement of the point-graph of a generalized quadrangle of order $(s,t)$ is a strongly regular graph with parameters $((s+1)(st+1), s^2t, t(s^2+1)-s(t+1), st(s-1))$. Its distinct eigenvalues are $s^2t, t$, and $-s$ with respective multiplicities $1, \frac{st(s+1)(t+1)}{s+t}, \frac{s^2(st+1)}{s+t}$. Brouwer \cite{Br1} observed that the toughness of the complement of the point-graph of a generalized quadrangle of order $(q,q)$ is between $q^2-2$ and $q^2$. Our next theorem improves this result.
\begin{theorem}
If $G$ is the complement of the point-graph of a generalized quadrangle of order $(s,t)$, then $t(G)=st$. Moreover, $S$ is a disconnecting set of $G$ such that $\frac{|S|}{c(G\setminus S)}=st$ if and only if $S$ is the complement of a maximum independent set or $S$ is the neighborhood of a vertex (when $s=2$).  
\end{theorem}
\begin{proof}
Since a generalized quadrangle is triangle-free, the points of an independent set of size at least $3$ in $G$ must be on a line in $GQ(s,t)$. Hence, $\alpha(G)=s+1$, which meets the Hoffman ratio bound $\frac{n(-\lambda_{min})}{k-\lambda_{min}}=\frac{(s+1)(st+1)s}{s^2t+s}=s+1$. Taking $S$ to be the complement of a maximum independent set, we conclude that
\begin{equation}
t(G)\leq \frac{(s+1)(st+1)-(s+1)}{s+1} = st.
\end{equation}
Let $S$ be a disconnecting set of vertices of $G$. We claim that if $G\setminus S$ contains at least three components, then the components of $G\setminus S$ must be isolated vertices. Otherwise, $G\setminus S$ would contain an induced subgraph with two isolated vertices, $v_1$ and $v_2$, and an edge $v_3v_4$. This would imply that $v_1$ and $v_2$ are on two different lines in $GQ(s,t)$, a contradiction.  

Therefore, if $c(G\setminus S)\geq 3$, then $G\setminus S$ consists of at most $s+1$ isolated vertices. 
Thus, $\frac{|S|}{c(G\setminus S)}=\frac{(s+1)(st+1)-c(G\setminus S)}{c(G\setminus S)}\geq st$, with equality if and only if $S$ is the complement of an independent set of maximum size in $G$. The only remaining case is when $G\setminus S$ has exactly two components. Brouwer and Mesner \cite{BM} proved that the vertex-connectivity of any connected strongly regular graph equals its degree and the only disconnecting sets are neighborhoods of vertices. This implies that $|S|\geq s^2t$ when $c(G\setminus S)=2$. Thus, $\frac{|S|}{c(G\setminus S)}\geq \frac{s^2t}{2}\geq st$ with equality if and only $S$ is the neighborhood of a vertex and $s=2$. This finishes our proof.
\end{proof}
We recall here the fact (see e.g. \cite[Section 10.8]{GR}) that $s=2$ implies $t\in \{1,2,4\}$. The complement of the point-graph of $GQ(2,1)$ is the Lattice graph $L_2(3)$. The complement of the point-graph of $GQ(2,2)$ is the triangular graph $T_6$. The complement of the point-graph of $GQ(2,4)$ is the Schl\"{a}fli graph.

\section{Final Remarks}

In this paper, we improved some spectral lower bounds of Alon \cite{Al}, Brouwer \cite{Br} and Liu and Chen \cite{LC} for the toughness of a connected $k$-regular graph in certain ranges of the toughness. Following the suggestion of one of the referees, we remark here that while a connected $k$-regular graph with small $\lambda$ has large toughness, the converse is not true. One can construct easily circulant claw-free graphs (certain Cayley graphs of finite cyclic groups) that would have toughness equal to half their valency (by Mathews-Sumner result \cite{MS}) while $\lambda$ would be very close to $k$ (see e.g. \cite{CioabaComptes}).

We also showed that with the exception of the Petersen graph, the toughness of many strongly regular graphs attaining equality in the Hoffman ratio bound for the independence number, equals $k/(-\lambda_{min})$, thus confirming Brouwer's belief \cite{Br1}. It would be interesting to extend our results and determine the toughness of other strongly regular graphs such as Latin square graphs or block graphs of Steiner systems.

The Kneser graph $K(v,r)$ has as vertex set the $r$-subsets of set with size $v$ and two $r$-subsets are adjacent if and only if they are disjoint. The Erd\H{o}s-Ko-Rado theorem is a fundamental result in extremal graph theory stating that for $v\geq 2r+1$, the independence number of the Kneser graph $K(v,r)$ is ${v-1\choose r-1}$ and any independent set of maximum size consists of all $r$-subsets containing a fixed element. The independence number of the Kneser graph can be obtained using eigenvalue methods as 
the Kneser graphs $K(v,r)$ attain equality in the Hoffman ratio bound (see e.g. \cite[Sections 7.8 and 9.6]{GR}). We leave the detailed calculation of the toughness of Kneser graphs $K(v,r)$ for $v\geq 2r+1$ and $r\geq 3$ for a future work \cite{W1}. Another interesting problem would be determining whether or not there exists a regular graph except the Petersen graph, attaining equality in Hoffman ratio bound and having toughness not equal to $k/(-\lambda_{min})$.

\section*{Acknowledgments}

We are grateful to the anonymous referees whose comments and suggestions have greatly improved the quality of our paper. This work was partially supported by a grant from the Simons Foundation ($\#209309$ to Sebastian M. Cioab\u{a}) and by National Security Agency grant H98230-13-1-0267. This work was partially completed when the second author was a graduate student at the University of Delaware and is part of his Ph.D. Thesis \cite{WongThesis}.


\begin{thebibliography}{99}

\bibitem {Al} N. Alon, Tough Ramsey graphs without short cycles,
{\em J. Algebraic Combin.} {\bf 4} (1995), no. 3, 189--195.  

\bibitem {Ba} D. Bauer, H.J. Broersma and E. Schmeichel, Toughness of Graphs - A Survey, {\em Graphs Combin.} \textbf{22} (2006), 1--35. 

\bibitem {BBV} D. Bauer, H.J. Broersma and H.J. Veldman, Not every $2$-tough graph is Hamiltonian, {\em Discrete Appl. Math.} {\bf 99} (2000), 317--321.

\bibitem {BHMS} D. Bauer, J. van den Heuvel, A. Morgana and E. Schmeichel, The complexity of recognizing tough cubic graphs, {\em Discrete Appl. Math.}  {\bf 79} (1997), 35--44.

\bibitem {BHMS1} D. Bauer, J. van den Heuvel, A. Morgana and E. Schmeichel, The complexity of toughness in regular graphs, {\em Congr. Numer.} {\bf 130} (1998), 47--61.

\bibitem {Br} A.E. Brouwer, Toughness and spectrum of a graph, {\em Linear Algebra Appl.}  {\bf 226/228} (1995), 267--271. 

\bibitem {Br1}  A.E. Brouwer, Connectivity and spectrum of graphs, {\em CWI Quarterly} {\bf 9} (1996), 37--40.

\bibitem {BH} A.E. Brouwer and W.H. Haemers, {\em Spectra of Graphs}, Springer Universitext (2012), 250pp.

\bibitem {BM} A.E. Brouwer and D. Mesner, The connectivity for strongly regular graphs, {\em European J. Combin.} {\bf 6} (1985), 215--216.

\bibitem {Chv} V. Chv\'{a}tal, Tough graphs and hamiltonian circuits, {\em Discrete Math.} {\bf 5} (1973), 215--228.

\bibitem {CioabaComptes} S.M. Cioab\u{a}, Closed walks and eigenvalues of Abelian Cayley graphs, {\em C.R.Acad.Sci.Paris Ser.I Math} {\bf 342} (2006), 635--638.

\bibitem {CGH} S.M. Cioab\u{a}, D. A. Gregory, W. H. Haemers,  Matchings in regular graphs from eigenvalues, {\em J. Combin. Theory  Ser. B} {\bf 99} (2009), 287--297.

\bibitem {GR} C. Godsil and G. Royle, {\em Algebraic Graph Theory}, Graduate Texts in Mathematics 207, Springer Verlag, New York, 2001.

\bibitem {Ha} W. H. Haemers, Interlacing eigenvalues and graphs, {\em Linear Algebra Appl.} {\bf 226-228} (1995), 593-616.

\bibitem {LC}  B. Liu and S. Chen, Algebraic conditions for $t$-tough graphs, {\em Czechoslovak Math. J.} {\bf 60} (135), (2010), 1079--1089.

\bibitem {MS} M.M. Matthews and D.P. Sumner, Hamiltonian results in $K_{1,3}$-free graphs, {\em J. Graph Theory} {\bf 8} (1984), 139--146.

\bibitem {WongThesis} W. Wong, {\em Spanning trees, toughness and eigenvalues of regular graphs}, Ph.D. Thesis, Department of Mathematical Sciences, University of Delaware, USA, (2013), i-x+91pp.

\bibitem {W1} W. Wong, The toughness of Kneser graphs, manuscript in preparation.

\end{thebibliography}
\end{document}